\documentclass[12pt,reqno]{amsart}
\usepackage{amsmath,amsthm,amsfonts,amssymb,amscd,  multicol, verbatim, enumerate}
\usepackage{verbatim}
\usepackage{fullpage}
\input{xy}
\usepackage{xcolor}
\usepackage[english]{babel}
\usepackage{cite}
\usepackage{amsfonts}
\usepackage{latexsym}
\usepackage{amsthm}
\usepackage{amsopn}
\usepackage{amsmath}
\usepackage[all]{xy}
\usepackage{verbatim}
\usepackage{amssymb}
\usepackage{mathrsfs}
\usepackage{float}
\usepackage{tikz-cd}
\usepackage{tikz}
\usepackage{fullpage}
\usepackage{amscd}
\usepackage{amscd}

\usepackage{graphicx}

\newcommand{\CC}{\mathbb{C}}

\newcommand{\Hom}{\mathrm{Hom}}

\newcommand{\TT}{\mathbb{T}}

\newcommand{\ZZ}{\mathbb{Z}}

\newtheorem{thm}{Theorem}[section]
\newtheorem*{Theorem*}{Theorem}
\newtheorem*{Corollary*}{Corollary}

\newtheorem{lem}[thm]{Lemma}
\newtheorem{prop}[thm]{Proposition}
\newtheorem{cor}[thm]{Corollary}

\theoremstyle{definition}
\newtheorem{defn}[thm]{Definition}

\theoremstyle{remark}

\newtheorem*{rmk*}{Remark}

\title{Topological Tensor Representations of $\mathfrak{gl}(V)$ for a space $V$ of countable dimension}
\author[Esposito]{Francesco Esposito}
\address{Francesco Esposito, Dipartimento di Matematica, Universit\`a degli Studi di Padova,
via Trieste 63, 35121 Padova, Italy}
\email{esposito@math.unipd.it}
\author[I. Penkov]{Ivan Penkov}
\address{Ivan Penkov, Jacobs University Bremen, 28759 Bremen, Germany}
\email{i.penkov@jacobs-university.de}
\dedicatory{to Yuri Ivanovich Manin on the occasion of his 85th birthday}

\begin{document}
	
	\maketitle
	
	\begin{abstract}
		
		The Lie algebra $ \mathfrak{gl}(V) $ is the Lie algebra of all endomorphisms of a countable-dimensional complex vector space $ V $. We define a tensor category $\widehat{\mathbf{T}}_{\mathfrak{gl}(V)}$ of topological representations of the Lie algebra $ \mathfrak{gl}(V) $, so that $ V $, $ \mathbf{V} := \mathrm{Hom}_{\CC}(V, \CC) $, and the adjoint representation $ \mathfrak{gl}(V) $ are objects of $\widehat{\mathbf{T}}_{\mathfrak{gl}(V)}$. This makes  $\widehat{\mathbf{T}}_{\mathfrak{gl}(V)}$ an analogue of the category of finite-dimensional modules over the finite-dimensional Lie algebra $\mathfrak{gl}(n)$.
		
		Our main result is that the category $\widehat{\mathbf{T}}_{\mathfrak{gl}(V)}$ is antiequivalent as a symmetric monoidal category to the category $\TT_{\mathfrak{gl}(\infty)}$ of tensor representations of the Lie algebra $\mathfrak{gl}(\infty)$ of finitary infinite matrices. The category $\TT_{\mathfrak{gl}(\infty)}$ has been studied extensively and is known to have a universality property as a tensor category. 
		
		
	\end{abstract}
	
	\section{Introduction}
	Let $V$ be a countable-dimensional vector space over $\CC$. The Lie algebra $\mathrm{End}V$ of all linear endomorphisms of $ V $ is a most natural infinite-dimensional analog of the Lie algebra $\mathfrak{gl}(n)$ for $n\in\mathbb{Z}_{n>0}$. However, the theory of representations of the Lie algebra $ \mathrm{End}V $ is much less developed than the representation theories of other infinite-dimensional analogues of $ \mathfrak{gl}(n) $ such as Kac-Moody Lie algebras or the Lie algebra $ \mathfrak{gl}(\infty) $. In this paper, by $ \mathfrak{gl}(\infty) $ we denote the ``smallest'' analogue of $ \mathfrak{gl}(n) $ among  Lie algebras of infinite matrices, i.e., $ \mathfrak{gl}(\infty) $ is the direct limit $ \varinjlim \mathfrak{gl}(n) $ for $ n\to\infty $; equivalently, $ \mathfrak{gl}(\infty) $ is the Lie algebra of finitary infinite matrices (infinite matrices with at most finitely many nonzero entries). By $ \mathfrak{gl}(V) $ we denote the Lie algebra $ \mathrm{End}V$.
	
	It is known that the Lie algebra $ \mathfrak{gl}(V) $ has no nontrivial finite-dimensional representations, and we are interested in the following question: \textit{what is a category of $ \mathfrak{gl}(V) $-modules, rich enough to resemble the category of all finite-dimensional $ \mathfrak{gl}(n) $-modules}? This same question, with $ \mathfrak{gl}(V) $ replaced by $\mathfrak{gl}(\infty)$, has been studied and various answers have been provided: see \cite{Cat} and \cite{PH}.
	
	One possible answer in the case of $ \mathfrak{gl}(\infty) $ is the category of tensor modules $\TT_{\mathfrak{gl}(\infty)}$, or $\TT_{\mathfrak{sl}(\infty)}$, introduced in \cite{DPS}, see also \cite{PStyr} and \cite{SS}. This category is an abelian symmetric monoidal category of $ \mathfrak{gl}(\infty) $-modules generated by the two natural representations $ V $ and $ V_{*} $ of $ \mathfrak{gl}(\infty) $, each of them being the restricted dual of the other. The category $\TT_{\mathfrak{gl}(\infty)}$ has a number of remarkable properties, in particular it is a non-semisimple Koszul category which is universal among non-rigid linear symmetric monoidal abelian categories generated by two objects $ X, Y $ with a pairing $ X\otimes Y\to\mathbf{1} $, where $ \mathbf{1}$ denotes monoidal unit. The category $\TT_{\mathfrak{gl}(\infty)}$ has also found notable applications beyond the theory of representations of $ \mathfrak{gl}(\infty) $ \cite{EHS}.
	
	In the case of the Lie algebra $ \mathfrak{gl}(V) $, an analogue $\TT_{\mathfrak{gl}(V)}$ of the category $\TT_{\mathfrak{gl}(\infty)}$ has been introduced in \cite{Mac} and has been further studied in \cite{C}.  The category $\TT_{\mathfrak{gl}(V)}$ is an abelian symmetric monoidal category of $ \mathfrak{gl}(V) $-modules generated by the $ \mathfrak{gl}(V) $-modules $ V $ and $ \mathbf{V} $:= $ \mathrm{Hom}_{\CC}(V, \CC) $, and it is proved in \cite{C} that $\TT_{\mathfrak{gl}(V)}$ is equivalent as a tensor category to $\TT_{\mathfrak{gl}(\infty)}$. Consequently, all categorical properties of $\TT_{\mathfrak{gl}(\infty)}$ hold also for $\TT_{\mathfrak{gl}(V)}$, in particular $\TT_{\mathfrak{gl}(V)}$ is universal in the same sense as $\TT_{\mathfrak{gl}(\infty)}$. 
	
	Still, $\TT_{\mathfrak{gl}(V)}$ does not qualify for a good answer to our question above, since the adjoint representation of $ \mathfrak{gl}(V) $ is not an object of $\TT_{\mathfrak{gl}(V)}$. Indeed, note that $ V\otimes\mathbf{V} $ is merely a submodule of $ \mathfrak{gl}(V) $ and the quotient $\mathfrak{gl}(V)/(V\otimes\mathbf{V})$ is not isomorphic to a subquotient of a direct sum of modules of the form $V^{\otimes p}\otimes \mathbf{V}^{\otimes q}$ for $p,q\in\mathbb{Z}_{\geq 0}$.
	On the other hand, 
	$\mathfrak{gl}(\infty) \simeq V\otimes V_{*}$, so the adjoint representation of $\mathfrak{gl}(\infty)$ is an object of $\TT_{\mathfrak{gl}(\infty)}$. 
	
	This is why we propose another answer to our central question. We construct a category $\widehat{\mathbf{T}}_{\mathfrak{gl}(V)}$ of topological $ \mathfrak{gl}(V) $-modules so that all three modules $ V, \mathbf{V} $, and $ \mathfrak{gl}(V) $ are objects of $\widehat{\mathbf{T}}_{\mathfrak{gl}(V)}$. While as an algebraic $ \mathfrak{gl}(V) $-module $ \mathfrak{gl}(V) $ has length 4, see \cite{BHB}, \cite{Mac} and \cite{HZ}, in the topological category $\widehat{\mathbf{T}}_{\mathfrak{gl}(V)}$ the object $ \mathfrak{gl}(V) $ has length 2, its only proper subobject being the trivial submodule generated by the identity operator.
	
	Our main result is that the category $\widehat{\mathbf{T}}_{\mathfrak{gl}(V)}$ is antiequivalent to the category $\TT_{\mathfrak{gl}(\infty)}$. Consequently, the category $\widehat{\mathbf{T}}_{\mathfrak{gl}(V)}$ is endowed with a symmetric tensor product $\widehat{\otimes}^!$, and moreover is a universal symmetric monoidal category. Its universality property is the ``mirror image'' of the universality property of the category $\TT_{\mathfrak{gl}(\infty)}$. The tensor product $\widehat{\otimes}^!$ is a particular instance of a tensor product introduced by A. Beilinson in \cite{Bei}.
	
	Here is a brief description of the contents of the paper. Section 2 is devoted to topological preliminaries, mainly to ind-linearly compact vector spaces and pro-discrete vector spaces. These types of spaces form two respective quasi-abelian semisimple categories which are antiequivalent. 
	In Section 3, we present our main objects of study: the categories $\mathbf{T}_{\mathfrak{gl}(V)}$ and $\widehat{\mathbf{T}}_{\mathfrak{gl}(V)}$. The former category is a completion of the category $\mathbb{T}_{\mathfrak{gl}(V)}$ and its objects are closed $ \mathfrak{gl}(V) $-subquotients of finite direct sums of tensor products of the form $ V^{\otimes p}\otimes\mathrm{Hom}_{\CC}(V^{\otimes q}, \CC)$. 
	We show that $\mathbf{T}_{\mathfrak{gl}(V)}$ and $\mathbb{T}_{\mathfrak{gl}(\infty)}$ are equivalent as abelian categories.
	The category $\widehat{\mathbf{T}}_{\mathfrak{gl}(V)}$ is defined as the continuous dual of the category $\mathbf{T}_{\mathfrak{gl}(V)}$, and it is antiequivalent to $\mathbf{T}_{\mathfrak{gl}(V)}$ as an abelian category.

	In Section 4 we introduce respective tensor products $\widehat{\otimes}^\ast$ and $\widehat{\otimes}^!$ on the categories $\mathbf{T}_{\mathfrak{gl}(V)}$ and $\widehat{\mathbf{T}}_{\mathfrak{gl}(V)}$, and show that in fact there is an antiequivalence of abelian symmetric monoidal categories $(\mathbf{T}_{\mathfrak{gl}(V)}, \widehat{\otimes}^\ast)$ and $(\widehat{\mathbf{T}}_{\mathfrak{gl}(V)}, \widehat{\otimes}^!)$. We then
	draw some corollaries of the antiequivalence of the categories $ \mathbf{T}_{\mathfrak{gl}(V)} $ and $\widehat{\mathbf{T}}_{\mathfrak{gl}(V)}$. In particular, objects of the form $ V^{\otimes p}\widehat{\otimes}^!\mathbf{V}^{\widehat{\otimes}^! q}$ are projective and $ \mathfrak{gl}(V) \simeq V\widehat{\otimes}^!\mathbf{V}$. Consequently, the object $ \mathfrak{gl}(V) $ (adjoint representation) is a projective cover of the simple object  $ \mathfrak{gl}(V)/\CC $. We also obtain explicit fomulas for the radical filtrations of indecomposable projectives, as well as for the Exts between simple objects of $\widehat{\mathbf{T}}_{\mathfrak{gl}(V)}$.
	
	\textbf{Acknowledgements.}
	This research has been partially supported by the project of the University of Padova BIRD203834/20.
	The second author thanks INdAM for facilitating a visit to the University of Padua in the fall of 2019, when some first collaboration plans between the suthors emerged.
	The second author acknowledges also partial support by the DFG through grant PE 980/8-1.
	
	\section{Preliminaries on topological vector spaces}
	In this section, we discuss the categories of topological vector spaces which play a role, and the relations between them.
	
	\subsection{Discrete and linearly compact topological vector spaces}
	The category of discrete vector spaces we consider is that of at most countable-dimensional vector spaces, with arbitrary linear maps as morphisms; i.e., we call a vector space $W$ \emph{discrete} if it is the union of a countable ascending family
	$$
	0 \subset W_1 \subset W_2 \subset \ldots
	$$
	of finite-dimensional vector spaces:
	$$
	W = \bigcup_{i\in \ZZ_{> 0}} W_i \ \ .
	$$
	
	The category of discrete vector spaces is abelian  and semisimple. The dual of a countable-dimensional vector space is in general not countable dimensional. Hence, taking (algebraic) duals does not preserve the category of discrete vector spaces, and taking the (algebraic) double dual does not yield back the vector space one started with. To get a well-behaved duality starting with discrete vector spaces, one endows their duals with the structure of \emph{linearly compact} vector spaces, which we now define.
	
	\begin{defn}
	  A topological vector space $Z$ is said to be \emph{linearly compact} if 
	\begin{enumerate}[(i)]
	    \item it has a fundamental system of open neighborhoods of $0$ given by a countable descending filtration of (open) subspaces of finite codimension
	    $$
	    Z = Z^{(0)} \supset Z^{(1)} \supset \ldots \ \ ;
	    $$
	    \item the canonical morphism 
	    $$
	    Z \rightarrow \varprojlim Z/Z^{(i)}
	    $$
	    is an isomorphism, where $\varprojlim Z/Z^{(i)}$ is endowed with the projective limit topology of the discrete topologies on the finite-dimensional vector spaces $Z/Z^{(i)}$.
	\end{enumerate}  
	\end{defn}
	
	A fundamental result about the category of linearly compact vector spaces with morphisms continuous linear maps, is that it is antiequivalent to the category of discrete vector spaces. Let us state this in the form of a lemma.
	
	\begin{lem} \label{lem:disc_lincomp}
    	 Let $W = \bigcup_i W_i$ be a discrete vector space. Then the dual space
    	$$W^\ast = \varprojlim (W_i)^\ast$$
    	is naturally a linearly compact topological vector space. Moreover, this construction is functorial and yields an antiequivalence between the category of discrete vector spaces and the category of linearly compact vector spaces.
    	\end{lem}
	\begin{proof}
	    Given a presentation $W = \bigcup_i W_i$ of a discrete vector space as a countable ascending union of finite-dimensional subspaces, one gets a presentation of the dual $W^\ast$ as the projective limit $\varprojlim (W_i)^\ast$. Thus $W^\ast$ may be endowed with the projective limit topology, making it a linearly compact vector space. In addition, one observes that different presentations of $W$ give rise to equivalent linearly compact topologies on $W^\ast$. 
	    
	    Furthermore, the continuous dual $Z^\ast$ of a linearly compact vector space 
	    $Z = \varprojlim Z/Z^{(i)}$ is naturally a discrete space which is the union of the finite-dimensional vector spaces $(Z/Z^{(i)})^\ast$; and this construction is independent from the presentation.
	    
	    The above two constructions are functorial and mutually inverse.
	\end{proof}
	
	Here are some easy consequences of Lemma~\ref{lem:disc_lincomp}.
	
	\begin{lem}\label{lem:lc}
	    The following statements hold:
	    \begin{enumerate}[(i)]
	        \item \label{closed2} a continuous linear map between linearly compact vector spaces is closed, i.e., takes closed subsets to closed subsets;
	        \item \label{lin_comp_closure} if $U$ be a subspace of a linearly compact vector space $Z = \varprojlim Z_i$, then its closure $\overline{U}$ equals $\varprojlim U_i$, where $U_i$ is the image of $U$ in $Z_i$;
	        \item \label{union2} a linearly compact vector space is not the union of a countable family of nested proper closed subspaces;
	        \item \label{lin_comp_ab} the category of linearly compact vector spaces, with continuous linear maps as morphisms, is a semisimple abelian category.
	    \end{enumerate}
	\end{lem}
	
	\begin{proof}
	    (\ref{closed2}) follows from the antiequivalence of Lemma~\ref{lem:disc_lincomp} and from the fact that such antiequivalence yields, by taking annihilators, an order-preserving  bijection between the closed subspaces of a linearly compact vector space and the subspaces of its continuous dual, the latter being a discrete vector space.
	    
	    To obtain (\ref{lin_comp_closure}), observe that $U$ and $\varprojlim U_i$ have the same annihilator in $Z^\ast$, hence, by Lemma~\ref{lem:disc_lincomp} and (\ref{closed2}), one gets $\overline{U}=\varprojlim U_i$.
	    
	    (\ref{union2}) follows from the antiequivalence of Lemma~\ref{lem:disc_lincomp} and from the observation that the projective limit of surjective linear maps of finite-dimensional vector spaces of unbounded dimensions is of uncountable dimension.
	    
	    (\ref{lin_comp_ab}) follows, by the antiequivalence of Lemma~\ref{lem:disc_lincomp}, from the fact that discrete spaces form a semisimple abelian category.
	\end{proof}
	
	\subsection{Ind-linearly compact and pro-discrete topological vector spaces}

    Here we present the definitions and main properties of the categories of ind-linearly compact and pro-discrete vector spaces. These categories are dual to one another and constitute natural ambient categories for the study of topological tensor representations.
	
	\begin{defn}
	    A topological vector space $W$ is said to be \emph{ind-linearly compact} if:
	    \begin{enumerate}[(i)]
	        \item it is the union of an ascending sequence of closed linearly compact subspaces
	        $$
	        W = \bigcup_i W_i \ \ ;
	        $$
	        \item the topology of $W$ is the inductive limit topology of the topologies of the subspaces $W_i$, i.e., a subset $U$ of $W$ is closed if and only if, for each $i$, the subset $U_i = U \cap W_i$ is closed in $W_i$.
	    \end{enumerate}
	    We denote $\mathcal{I}$ the category whose objects are ind-linearly compact vector spaces and whose morphisms are continuous linear maps.
	\end{defn}
	
	The dual definition is that of a pro-discrete vector space.
	
	\begin{defn}
	    A topological vector space $Z$ is said to be \emph{pro-discrete} if:
	    \begin{enumerate}[(i)]
	        \item it has a fundamental system of open neighborhoods of $0$ given by a countable descending filtration of (open) subspaces of countable codimension
	    $$
	    Z = Z^{(0)} \supset Z^{(1)} \supset \ldots \ \ ;
	    $$
	    \item the canonical morphism 
	    $$
	    Z \rightarrow \varprojlim Z/Z^{(i)}
	    $$
	    is an isomorphism, where $\varprojlim Z/Z^{(i)}$ is endowed with the projective limit topology of the discrete topologies on the  vector spaces $Z/Z^{(i)}$.
	    \end{enumerate}
	    
	    We denote $\mathcal{P}$ the category whose objects are pro-discrete vector spaces and whose morphisms are continuous linear maps.
	\end{defn}
	
	The categories $\mathcal{I}$ and $\mathcal{P}$ contain the abelian semisimple categories of discrete vector spaces and linearly compact vector spaces. However, $\mathcal{I}$ and $\mathcal{P}$ are only \emph{quasi-abelian} categories (see \cite{Schneiders}, Definition 1.1.3 or \cite{Buh}, Definition 4.1), i.e., are additive categories in which 
	\begin{enumerate}[(i)]
	    \item every morphism admits kernel and cokernel;
	    \item the push-out of a kernel is a kernel;
	    \item the pull-back of a cokernel is a cokernel.
	\end{enumerate}
	
	Let us recall some basic definitions and facts. A quasi-abelian category has a canonical \emph{exact} structure (see \cite{Schneiders}, 1.1.7), for which the short exact sequences are the kernel-cokernel pairs $(i,p)$ (i.e., $i$ is the kernel of $p$ and $p$ is the cokernel of $i$). 
	Furthermore, recall that a morphism in a quasi-abelian category is called \emph{strict} if the canonical morphism from its coimage to its image is an isomorphism, or equivalently, if it is the composition of a strict monomorphism (i.e., a kernel) after a strict epimorphism (i.e., a cokernel).
	
	We now prove some basic results about the categories $\mathcal{I}$ and $\mathcal{P}$. These are special cases of statements for more general topological vector spaces, but, for the convenience of the reader, we give independent and self-contained proofs.
	\begin{prop}\label{prop:duals}
	    The categories $\mathcal{I}$ and $\mathcal{P}$  are dual to one another under the respective functors of taking continuous duals. Moreover, duality sends strict morphisms to strict morphisms.
	\end{prop}
	\begin{proof}
	    Let us show that taking continuous duals yields well-defined contravariant functors from $\mathcal{I}$ to $\mathcal{P}$ and from $\mathcal{P}$ to $\mathcal{I}$, that are quasi-inverses of each other.
	    
	    Let $W = \bigcup_i W_i$ be a presentation of the ind-linearly compact vector space as union of an ascending sequence of closed linearly compact subspaces $W_i$. By the definition of inductive limit topology, a linear form $\varphi$ on $W$ is continuous if and only if all of its restrictions $\varphi_i = \varphi_{|W_i}$ are continuous. Therefore, the continuous dual $W^\ast$ of $W$ may be identified with the projective limit $\varprojlim W_i ^\ast$ of discrete spaces $W_i^\ast$. The continuous dual $W^\ast$ is hence naturally a pro-discrete vector space, and one obtains equivalent topologies for all presentations of $W$. It is straightforward to check that this yields a functor 
	    $$
	    (\phantom{a})^\ast : \mathcal{I} \rightarrow \mathcal{P}\ \ .
	    $$
	    
	    Let now $Z = \varprojlim Z/Z^{(i)}$ be a presentation of the pro-discrete vector space $Z$ as the projective limit of the discrete spaces $Z/Z^{(i)}$. By definition of the projective limit topology, a linear form $\varphi$ on $Z$ is continuous if and only if it factors through one of the discrete quotients $Z/Z^{(i)}$. Therefore, the continuous dual $Z^\ast$ may be identified with the union, i.e., the inductive limit, of the duals $(Z/Z^{(i)})^\ast$ of the discrete spaces $Z/Z^{(i)}$. It is hence naturally an ind-linearly compact space, and one gets equivalent topologies for all presentations of $Z$. Moreover, it is clear that a functor of continuous dual
	    $$
	    (\phantom{a})^\ast : \mathcal{P} \rightarrow \mathcal{I}
	    $$
	    is well defined.
	    
	    The fact that the above two functors of continuous dual are quasi-inverse to each other follows from the duality between discrete and linearly compact vector spaces stated in Lemma~\ref{lem:disc_lincomp}.
	    
	\end{proof}
	
	The next result is a special case of the open mapping theorem.
\begin{lem}\label{lem:iso-ind}
        \begin{enumerate}[(a)]
            \item Let $f:W\rightarrow Z$ be a continuous linear bijection of ind-linearly compact vector spaces. Then $f$ is bicontinuous, i.e., is an isomorphism in $\mathcal{I}$.
            \item Let $g:W'\rightarrow Z'$ be a continuous linear bijection of pro-discrete vector spaces. Then $g$ is bicontinuous, i.e., is an isomorphism in $\mathcal{P}$.
        \end{enumerate}
		 
	\end{lem}
	
	\begin{proof}
		
		Let $W = \bigcup_i W_i$ and $Z = \bigcup_j Z_j$ be presentations of the ind-linearly compact vector spaces $W$ and $Z$ as inductive limits of linearly compact closed subspaces.
		
		By the continuity of $f$, the subspaces $f^{-1}(Z_j)$ are closed in $W$. Moreover one has 
		$$
		W_i = \ \bigcup_j \ (W_i \cap f^{-1}(Z_j)).
		$$
		It follows that, by \ref{lem:lc}~(\ref{union2}), for every $i\in\mathbb{Z}_{>0}$ there is $j\in\mathbb{Z}_{>0}$ such that
		$$
		W_i \subset f^{-1}(Z_j).
		$$
		Furthermore, the restriction $f_{|W_i}: W_i \rightarrow Z_j$ is a continuous linear map between linearly compact vector spaces. Consequently, by \ref{lem:lc}~(\ref{closed2}), its image $f(W_i)$ is a closed subspace of $Z_j$, and thus of $Z$. This implies that the space $f(W_i)$ is linearly compact. Again by \ref{lem:lc}~(\ref{union2}), from 
		$$
		Z_j = \bigcup_i \ (Z_j \cap f(W_i))
		$$
		one infers that, for every $j\in\mathbb{Z}_{>0}$, there is $i\in\mathbb{Z}_{>0}$ such that $Z_j \subset f(W_i)$.
		
		From $W_i \subset f^{-1}(Z_j)$ and $Z_j \subset f(W_i)$ it is straightforward to deduce that the two inductive limit topologies are equivalent. This proves (a).
		
		To prove (b), observe that if $f = g^\ast$, then $g$ is continuous and bijective if and only if $f$ is continuous and bijective, and $g$ is an isomorphism in $\mathcal{P}$ if and only if $f$ is an isomorphism in $\mathcal{I}$. Therefore statements (a) and (b) are dual of each other under the duality of Proposition~\ref{prop:duals}, and hence are equivalent.
	\end{proof}
	
	\begin{lem}\label{lem:ind-linsubquot}
	        Let $Z$ be a closed subspace of the ind-linearly compact vector space $W$. Then $Z$, endowed with the subspace topology, and $W/Z$, endowed with the quotient topology, are ind-linearly compact topological vector spaces. Moreover, one has an isomorphism 
	        $$W\cong Z\oplus W/Z$$ 
	        in the category $\mathcal{I}$.
	\end{lem}
	\begin{proof}
	    Let $W = \bigcup_i W_i$ be a presentation of $W$ as a countable nested union of linearly compact closed subspaces, such that the topology on $W$ is the inductive limit topology. Then
	    $$
	    Z = \bigcup_i Z_i \ ,
	    $$
	    where $Z_i = Z \cap W_i$. It is straightforward to check that the subspace topology on $Z$ coincides with the inductive limit topology. Hence $Z$ is ind-linearly compact.
	    
	    Let us now discuss the quotient $W/Z$. Since $W_i$ and $Z_i$ are linearly compact, one has
	    $$
	    W/Z = \bigcup_i W_i / Z_i
	    $$
	    where the quotient topology on $W_i/Z_i$ is linearly compact. Furthermore, one can verify that the quotient topology on $W/Z$ coincides with the inductive limit topology of the topologies on $W_i / Z_i$. This proves that the space $W/Z$ is ind-linearly compact.
	    
	   Since the category of linearly compact vector spaces is abelian semisimple, it follows that for every $i$ one has a splitting $W_i \cong Z_i \oplus W_i/Z_i$. Moreover, proceeding inductively, one may choose all these splittings compatible, so that, taking the inductive limit, they give rise to an isomorphism
	   $$
	   W \cong Z\oplus W/Z
	   $$
	   in $\mathcal{I}$.
	\end{proof}
	
	\begin{prop}\label{prop:quasiab}
	    The categories $\mathcal{I}$ and $\mathcal{P}$ are antiequivalent quasi-abelian semisimple categories. Moreover, in these categories, kernels are closed embeddings and cokernels are surjective linear maps.
	\end{prop}
	
	\begin{proof}
	    Since the duality functor is additive, it follows from Proposition~\ref{prop:duals} that the categories $\mathcal{I}$ and $\mathcal{P}$ are antiequivalent as additive categories. To prove the proposition, it  is thus enough to show that $\mathcal{I}$ is quasi-abelian. Observe that $\mathcal{I}$ is a full subcategory of the category $\mathrm{Top}_\CC$ of topological vector spaces with linear topology (see \cite{Po}, Section 7). Furthermore, by \cite{Po}, Theorem 7.1(b), the category $\mathrm{Top}_\CC$ is quasi-abelian, its kernels are closed embeddings and its cokernels are open surjections. It follows hence, by Lemma~\ref{lem:ind-linsubquot}, that if $f:W\rightarrow Z$ is a morphism of $\mathcal{I}$, then its kernel and cokernel in $\mathrm{Top}_\CC$ are in fact objects of $\mathcal{I}$. Moreover,  pushing out or pulling back a morphism of $\mathcal{I}$ by a morphism of $\mathcal{I}$ affords a morphism of $\mathcal{I}$; it thus follows that the category $\mathcal{I}$ is quasi-abelian.
	    
	    The semisimplicity of $\mathcal{I}$ follows from the second part of Lemma~\ref{lem:ind-linsubquot}; the semisimplicity of $\mathcal{P}$ follows by duality. To conclude the proof, observe that Lemma~\ref{lem:iso-ind} implies that every continuous surjective linear map between two ind-linearly compact or two pro-discrete vector spaces is automatically open. This implies that the cokernels in $\mathcal{I}$ or $\mathcal{P}$ are the continuous linear surjections. Since we have already remarked that the kernels are the closed embeddings, this concludes the proof.
	\end{proof}

	\begin{lem}\label{lem:strict}
	Let $f$ be a morphism in either $\mathcal{I}$ or $\mathcal{P}$. Then the following statements are equivalent:
	\begin{enumerate}[(a)]
	    \item $f$ is strict;
	    \item $f$ has closed image.
	\end{enumerate}
	\end{lem}
	\begin{proof}
	Follows from the description of kernels and cokernels in $\mathcal{I}$ and $\mathcal{P}$ given in Proposition~\ref{prop:quasiab}.
	\end{proof}

	\section{The categories $\mathbf{T}_{\mathfrak{gl}(V)}$ and $\widehat{\mathbf T}_{\mathfrak{gl}(V)}$}
	
	In this section we introduce our main object of study, the category $\widehat{\mathbf T}_{\mathfrak{gl}(V)}$. This is a category of topological representations of the Lie algebra $\mathfrak{gl}(V)$, and we define it as the dual of a completion $\mathbf{T}_{\mathfrak{gl}(V)}$ of the category $\mathbb{T}_{\mathfrak{gl}(\infty)}$.

	\subsection{Topological vector spaces of mixed tensors}
	
	We fix a countable-dimensional vector space $V$ and a basis $\mathcal{V} = \{v_i\}_{i\in\mathbb{Z}_{>0}}$ of $V$. Then $V = \bigcup_i V_i$, where $V_i$ is the span of the first $i$ basis vectors. We consider $ V $ as a discrete topological vector space. Therefore the Lie algebra $\mathfrak{gl}(V)$ of linear endomorphisms of $V$ coincides with the Lie algebra of continuous endomorphisms of $V$ as a topological vector space. Inside the Lie algebra $\mathfrak{gl}(V)$ there is the infinite-dimensional subalgebra $\mathfrak{h}$ of endomorphisms of $V$ which have diagonal matrices with respect to the basis $\mathcal{V}$. We denote by $\mathbf{V}$ the space $ \mathrm{Hom}_{\CC}(V, \CC) $ which, of course, is also the continuous dual of $V$. Since $V = \bigcup_i V_i$, it follows that 
	$$
	\mathbf{V} = \varprojlim V_i ^\ast,
	$$
	and thus $ \mathbf{V} $ can be considered as a linearly compact vector space.
	Analogously, we consider the space $V^{\otimes p} = \bigcup_i V_i ^{\otimes p}$ as a discrete topological vector space and its dual $(V^{\otimes p})^{\ast} = \mathrm{Hom}_{\mathbb{C}}(V^{\otimes p}, \mathbb{C})$ as a linearly compact topological vector space.
	
	In what follows, we use the superscript $^\ast$ to indicate continuous dual space. The precise meaning will depend on the context, as we can apply $^\ast$ to a discrete, linearly compact, ind-linearly compact, or pro-discrete vector space.

	Next, we define the spaces $\mathbf{V}^{p,q}$ as $V^{\otimes p}\otimes (V^{\otimes q})^*$. These are naturally ind-linearly compact vector spaces. Indeed,
	$$
	\mathbf{V}^{p,q} = \bigcup_i \ (V_i ^{\otimes p}\otimes (V^{\otimes q})^*),
	$$
	which yields a presentation of $\mathbf{V}^{p,q}$ as an ascending union of linearly compact vector spaces, since the spaces $V_i ^{\otimes p}$ are finite dimensional and the spaces $(V^{\otimes q})^*$ are linearly compact.
	
	Other relevant topological vector spaces are obtained by taking continuous duals of the spaces of mixed tensors just defined. More precisely, we set 
	$$
	\widehat{\mathbf{V}}^{p,q}
	:= 
	(\mathbf{V}^{q,p})^{\ast}.
	$$
	The structure of pro-discrete topological vector space on $ \widehat{\mathbf{V}}^{p,q}$ may be seen explicitly as follows:
	$$
	\widehat{\mathbf{V}}^{p,q} = \Hom_{\mathbb{C}} (\varinjlim V_i^{\otimes q}, V^{\otimes p}) = 
	\varprojlim \Hom_{\mathbb{C}} ( V_i ^{\otimes q},  V^{\otimes p}) = \varprojlim \ ( V_i ^{\ast}\otimes V^{\otimes p}).
	$$
	
	Observe that both topological vector spaces $V$ and $\mathbf{V}$ have obvious  $\mathfrak{gl}(V)$-module structures, therefore the spaces of mixed tensors $\mathbf{V}^{p,q}$ and $\widehat{\mathbf{V}}^{p,q}$ are $ \mathfrak{gl}(V) $-modules.
	
	\subsection{Structure of $\mathbf{V}^{p,q}$ as an $\mathfrak{h}$-module}
	Recall that an element $\chi\in \mathfrak{h}^\ast$ is called a weight, and the $\chi$-weight space of a $\mathfrak{gl}(V)$-module $W$ is the subspace
	$$
	W^\chi = \left\{ w\in W \ | \ \forall t\in\mathfrak{h}, \ tw = \chi(t)w \right\} \ .
	$$
	The sum of all one-dimensional $\mathfrak{h}$-submodules of $W$ is the \textit{weight part} $W^{wt}$ of $W$; it is the maximal semisimple $\mathfrak{h}$-submodule of $W$.
	
	\begin{lem} \label{lem: weightpart}
	The following statements hold:
	\begin{enumerate}
	    \item[(a)]\label{product}
	    The subspaces $V_i ^{\otimes p}\otimes (V^{\otimes q})^*$ are $\mathfrak{h}$-submodules of the $\mathfrak{gl}(V)$-module $\mathbf{V}^{p,q}$. For any weight $\chi\in \mathfrak{h}^\ast$, the weight space $(V_i ^{\otimes p}\otimes (V^{\otimes q})^*)^\chi$ is finite dimensional, and 
	$$
	V_i ^{\otimes p}\otimes (V^{\otimes q})^* = \prod_{\chi} (V_i ^{\otimes p}\otimes (V^{\otimes q})^*)^\chi
	$$
	with only a countable number of weights occurring. Furthermore, 
	
	$$(\mathbf{V}^{p,q})^{wt} = V^{p,q} := V^{\otimes p} \otimes (V_\ast)^{\otimes q}  $$ 
	where $V_\ast = (V^\ast)^{wt}$.
	
	\item[(b)] \label{bij}
	There is an order-preserving bijection between closed $\mathfrak{h}$-stable subspaces of $\mathbf{V}^{p,q}$ and $\mathfrak{h}$-stable subspaces of ${V}^{p,q} := (\mathbf{V}^{p,q})^{wt}$, given in one direction by taking
	weight part, and in the other direction by taking closure. In particular $(\mathbf{V}^{p,q})^{wt}$ is dense in
	$\mathbf{V}^{p,q}$.
	\end{enumerate}

	\end{lem}
	
	\begin{proof} 
	    Recall that $\mathcal{V} = \{v_i\}_{i\in\mathbb{Z}_{>0}}$ is the fixed basis of $V$ with respect to which $\mathfrak{h}$ is the subalgebra of diagonal matrices. Then, for each $i\in\mathbb{Z}_{>0}$, the vector $v_i$ is of weight $\varepsilon_i$, where $\varepsilon_i$ is the projection to the $i$-th component of $\mathfrak{h}$. So the weights of $V$ are all $\varepsilon_i$ for $i>0$. Furthermore, the linear forms $v_i ^\ast$ form a topological basis of $\mathbf{V}$ which is dual to $\mathcal{V}$, and $v_i ^\ast$ has weight $-\varepsilon_i$. 
	    
	    Let us now study the weight spaces of $V_i ^{\otimes p}\otimes (V^{\otimes q})^*$. Observe that the space $V_i ^{\otimes p}\otimes (V^{\otimes q})^*$ is linearly compact and may be realized as the following inverse limit
	    $$
	    V_i ^{\otimes p}\otimes (V^{\otimes q})^* = 
	    \varprojlim_j 	V_i ^{\otimes p}\otimes (V_j ^{\otimes q})^* \ .
	    $$
	    The canonical surjection 
	    $$
	    V_i ^{\otimes p}\otimes (V_j ^{\otimes q})^* \rightarrow
	    	V_i ^{\otimes p}\otimes (V_{j-1} ^{\otimes q})^*
	    $$
	    is a surjective homomorphism of $\mathfrak{h}$-modules which has a unique right inverse due to the fact that it is an isomorphism when restricted to the direct sum of the weight subspaces of 
	    $V_i ^{\otimes p}\otimes (V_j^{\otimes q})^*$ corresponding to weights of 
	    $V_i ^{\otimes p}\otimes (V_{j-1}^{\otimes q})^*$. It follows that all weight spaces of 
	    $V_i ^{\otimes p}\otimes (V^{\otimes q})^*$ are finite dimensional. Furthermore, the canonical map
	    $$
	    \prod_\chi (V_i ^{\otimes p}\otimes (V^{\otimes q})^*)^\chi =
	    \varprojlim_\chi (V_i ^{\otimes p}\otimes (V^{\otimes q})^* )^\chi
	    \rightarrow 
	    \varprojlim_j V_i ^{\otimes p}\otimes (V_j^{\otimes q})^*
	    $$
	    	is an isomorphism of topological vector spaces. The weights of 
	$V_i ^{\otimes p}\otimes (V^{\otimes q})^*$ are of the form $\chi = \sum_{j=1}^{\infty}n_j \varepsilon_j$, where the coefficients $n_j$ are integers satisfying $-q \leq n_j \leq p$ for $j\leq i$, $-q \leq n_j \leq 0$ for $j > i$, and $\sum_j n_j = p-q$. 
	Furthermore, the above implies
	\begin{equation*}
	    \begin{split}
	    (\mathbf{V}^{p,q})^{wt} & = \bigcup_i (V_i ^{\otimes p}\otimes (V^{\otimes q})^*)^{wt} =
	\bigcup_{i,j} (V_i ^{\otimes p}\otimes (V_j^{\otimes q})^*)^{wt} \\
	& = 
	\bigcup_{i,j} (V_i ^{\otimes p}\otimes (V_j^{\otimes q})^*) 
	=  V ^{\otimes p}\otimes V_\ast^{\otimes q} = {V}^{p,q} \ ,
	\end{split}
	\end{equation*}
		where the injection 
	$V_i ^{\otimes p}\otimes (V_{j-1}^{\otimes q})^*
	\hookrightarrow
	V_i ^{\otimes p}\otimes (V_j^{\otimes q})^*$ is the unique homomorphism of $\mathfrak{h}$-modules which is right inverse to the canonical surjection 
	$V_i ^{\otimes p}\otimes (V_j^{\otimes q})^*
	\rightarrow
	V_i ^{\otimes p}\otimes (V_{j-1}^{\otimes q})^*$.
	This concludes the proof of (a).
	
	Let us now prove (b). Observe that, for each $i\geq 1$, there is an order-preserving bijection between the closed $\mathfrak{h}$-stable subspaces of $\prod_\chi (V_i ^{\otimes p}\otimes (V^{\otimes q})^*)^\chi$ and the $\mathfrak{h}$-stable subspaces of its weight part $\bigoplus_\chi (V_i ^{\otimes p}\otimes (V^{\otimes q})^*)^\chi$. If $A$ is an $\mathfrak{h}$-stable subspace of ${V}^{p,q}$, then $A = \bigcup_i A_i$ where $A_i = A \cap (V_i ^{\otimes p}\otimes (V^{\otimes q})^*)^{wt}$. 
	Furthermore, one checks that $\overline{A_i} = \overline{A_{i+1}}\cap (V_i ^{\otimes p}\otimes (V^{\otimes q})^*)$.
	Therefore the closure $\overline{A}$ of $A$ in $\mathbf{V}^{p,q}$ is equal to $\bigcup_i \overline{A_i}$, where $\overline{A_i}$ is the closure of $A_i$ in $V_i ^{\otimes p}\otimes (V^{\otimes q})^*$. Thus, taking the weight part gives rise to an order-preserving bijection between closed $\mathfrak{h}$-stable subspaces of $\mathbf{V}^{p,q}$ and $\mathfrak{h}$-stable subspaces of $V^{p,q}$. The inverse bijection is given by taking the closure. In particular, one has $\overline{(\mathbf{V}^{p,q})^{wt}}=\mathbf{V}^{p,q}$, so $(\mathbf{V}^{p,q})^{wt}$ is dense in $\mathbf{V}^{p,q}$. This concludes the proof of (b).
	\end{proof}

	\subsection{The category $\mathbf{T}_{\mathfrak{gl}(V)}$}
	We now define a completion of the category $\mathbb{T}_{\mathfrak{gl}(V)}$ from \cite{C}.

	\begin{defn}\label{def:tglv}
		The objects of the category $\mathbf{T}_{\mathfrak{gl}(V)}$ are topological vector spaces of the form $Z/U$, where $Z$ is a closed $\mathfrak{gl}(V)$-stable subspace of a finite direct sum of topological vector spaces of the form $\mathbf{V}^{p,q}$, and $U$ is a closed $\mathfrak{gl}(V)$-subspace of $Z$. The morphisms in the category $\mathbf{T}_{\mathfrak{gl}(V)}$ are $\mathfrak{gl}(V)$-equivariant continuous linear maps.
	\end{defn}
	
		\begin{lem}\label{T-split}
		Let $W$ be an object of $\mathbf{T}_{\mathfrak{gl}(V)}$ and let $Z$ be a closed $\mathfrak{h}$-stable subspace of $W$. Then there is a closed $\mathfrak{h}$-stable subspace $U$ of $W$ such that the canonical map $Z\oplus U \rightarrow W$ is an $\mathfrak{h}$-equivariant topological isomorphism. 
	\end{lem}
	
	\begin{proof}	    
	    It is clear from Definition~\ref{def:tglv} that it is enough to consider the case when $W = \mathbf{V}^{p,q}$. Here, the filtration $V_i ^{\otimes p}\otimes (V^{\otimes q})^*$ of $\mathbf{V}^{p,q}$ induces an $\mathfrak{h}$-stable filtration by closed linearly compact subspaces 
	    $W = \bigcup_i W_i$ which, by Lemma~\ref{lem: weightpart}(a), are isomorphic to direct products of finite-dimensional weight spaces. Analogously, the closed subspace $Z$ has the induced filtration $Z = \bigcup_i Z_i$.  For every weight $\chi$ of $W$, one may choose compatible supplementary subspaces $(U_i)^\chi$ of $(Z_i)^\chi$ in $(W_i)^\chi$. Then, the subspaces $U_i := \prod_\chi (U_i)^\chi$ are closed, $\mathfrak{h}$-stable, and such that there are isomorphisms $Z_i \oplus U_i \cong W_i$ of topological $\mathfrak{h}$-modules. Consequently, denoting by $U$ the union of the $U_i$, it follows that $U$ is closed and $\mathfrak{h}$-stable in $W$, and the canonical morphism $Z\oplus U \rightarrow W$ is an isomorphism of topological $\mathfrak{h}$-modules. This concludes the proof.
	\end{proof}

	Below we show that taking the weight part yields an equivalence of categories  
	$$(\phantom{a})^{wt}: \mathbf{T}_{\mathfrak{gl}(V)} \rightarrow \mathbb{T}_{\mathfrak{gl}(\infty)}\ . $$
	As a first step, we have
	
	\begin{lem}
	    Taking the weight part is a functor  $(\phantom{a})^{wt}: \mathbf{T}_{\mathfrak{gl}(V)} \rightarrow \mathbb{T}_{\mathfrak{gl}(\infty)}$.
	\end{lem}
	\begin{proof}
	    By Lemma~\ref{lem: weightpart}, the weight part of $\mathbf{V}^{p,q}$ is $V^{p,q}$. Note that $V^{p,q}$ is not $\mathfrak{gl}(V)$-stable, but is stable by the dense Lie subalgebra $\mathfrak{gl}(\infty)$ consisting of finitary matrices with respect to the basis $\mathcal{V}$ of $V$. Let $W\cong Z/U$ be an object of $\mathbf{T}_{\mathfrak{gl}(V)}$, with $U\subset Z$ two closed $\mathfrak{gl}(V)$-stable subspaces of a finite direct sum  $\bigoplus_i \mathbf{V}^{p_i,q_i}$. Then the weight parts $U^{wt}$ and $Z^{wt}$ are $\mathfrak{gl}(\infty)$-stable subspaces of $\bigoplus_i V^{p_i,q_i}$. Furthermore, Lemma~\ref{T-split} implies that $W^{wt} \cong Z^{wt}/U^{wt}$. Thus the weight part $W^{wt}$ is an object of the category $\mathbb{T}_{\mathfrak{gl}(\infty)}$. Next, it is clear that morphisms in $\mathbf{T}_{\mathfrak{gl}(V)}$ restrict to $\mathfrak{gl}(\infty)$-equivariant morphisms between weight parts. This proves the statement.
	\end{proof}

	\begin{lem}\label{lem:full}
	The map 
	$$\Phi:\Hom_{\mathbf{T}_{\mathfrak{gl}(V)}}(\mathbf{V}^{p,q}, \mathbf{V}^{p', q'})
	\rightarrow
	\Hom_{\mathbb{T}_{\mathfrak{gl}(\infty)}}((\mathbf{V}^{p,q})^{wt}, (\mathbf{V}^{p', q'})^{wt})
	$$
	induced by the functor $(\phantom{a})^{wt}$, is an isomorphism of vector spaces.
	\end{lem}
	
	\begin{proof}
	Since $(\mathbf{V}^{p,q})^{wt}$ is dense in $\mathbf{V}^{p,q}$, the map $\Phi$ is injective. Furthermore, by \cite{DPS} (see Lemma~6.1 and the text right above Lemma~6.1), the vector space $\Hom_{\mathbb{T}_{\mathfrak{gl}(\infty)}}((\mathbf{V}^{p,q})^{wt}, (\mathbf{V}^{p', q'})^{wt})$ is generated by compositions of contractions and permutations. These morphisms clearly extend  to $\mathbf{V}^{p,q}$ by continuity, and hence are in the image of the map $\Phi$. Thus $\Phi$ is also surjective.
	\end{proof}
	
	\begin{prop}\label{prop:strict}
	    Let $$W = \bigoplus_{k=1}^r \mathbf{V}^{p_k , q_k} \ \ \mathrm{and} 
	    \ \ W' = \bigoplus_{k'=1}^{r'} \mathbf{V}^{p'_{k'} , q'_{k'}}. $$
	    Then every $f\in\Hom_{\mathbf{T}_{\mathfrak{gl}(V)}}(W, W')$ is strict.
	    Moreover, if $U$ is a closed $\mathfrak{gl}(V)$-stable subspace of $W$, also the restriction $f_{|U}$ of $f$ to $U$ is strict.
	\end{prop}
	
	\begin{proof}
	Let $f\in\Hom_{\mathbf{T}_{\mathfrak{gl}(V)}}(W, W')$ and let $Z$ denote the image of $f$. By Lemma~\ref{lem:strict}, one must prove that $Z$ is a closed subspace of $W'$. 
	
	Let $W'_j = \bigoplus_{k'=1}^{r'} V_j ^{\otimes p'_k}\otimes (V^{\otimes q'_k})^*$ and $W_i = \bigoplus_{k=1}^r V_i ^{\otimes p_k}\otimes (V^{\otimes q_k})^*$. The topologies on $W$ and $W'$ are the inductive limit topologies on $\bigcup_i W_i$ and $\bigcup_j W'_j$ respectively. Thus, if $Z_j = Z \cap W'_j$, one has $Z = \bigcup_j Z_j$; moreover $Z$ is closed in $W'$ if and only if, for every $j$, the subspace $Z_j$ is closed in $W'_j$.
	
	Let us fix $j$ and prove that $Z_j$ is closed in $W'_j$. Set $Z_{j, i}:=f(W_i)\cap W'_j$. Then $Z_j = \bigcup_i Z_{j, i}$, and, by Lemma~\ref{lem:lc}(\ref{union2}), the subspace $Z_j$ is closed in $W'_j$ if and only if there exists $i$ such that $Z_{j, i} = Z_j$.
	
	We now show that there is $i$ such that $Z_{j, i} = Z_j$, thus proving the proposition. By Lemma~\ref{lem: weightpart}(a), the linearly compact vector spaces $W_i$ and $W'_j$ are such that, for any $\chi\in\mathfrak{h}^\ast$, the $\chi$-weight space is finite dimensional and
	$$
	\overline{Z_j}= \prod_\chi (Z_j)^\chi  \ \ \mathrm{and} \ \ Z_{j, i} =\prod_\chi (Z_{j, i})^\chi.
	$$
	
	Let $S$ denote the finitary symmetric group on countably many letters, i.e., $S=\bigcup_n S_n$.
	Then $S$ acts contragrediently on $\mathbf{V} = V^\ast$, thus also on the spaces $\mathbf{V}^{p, q}$, and hence on $W$ and $W'$. Lemma~\ref{lem:full} and \cite[Lemma 6.1]{DPS} imply that the map $f$ is $S$-equivariant. Observe that $W'_j$ is stable under the action of the subgroup $S'$ of $S$ which  fixes pointwise the set $[1,j]$. Furthermore, since $S$ normalizes $\mathfrak{h}$, it follows that the action of $S'$ permutes the weights and, accordingly, the respective weight spaces. Note that the weights $\chi$ for which $(W'_j)^{\chi}\neq 0$ form finitely many $S'$-orbits and each orbit contains a weight of $V_j ^{\otimes p}\otimes (V_{j+q}^{\ast})^{\otimes q}$. Let $i$ be such that, for every weight $\chi$ of 
	$V_j ^{\otimes p}\otimes (V_{j+q}^{\ast})^{\otimes q}$, one has $(Z_j)^\chi = (Z_{j, i})^\chi$.
	
	Let us show that this implies $(Z_j)^\psi = (Z_{j, i})^\psi$ for any weight $\psi$, and thus ultimately $Z_j = Z_{j, i}$. Indeed, let $E_{k,l}\in\mathfrak{gl}(V)$ be the endomorphism sending $v_l$ to $v_k$, and sending all other basis vectors to $0$. Observe that if $k<l$, one has $E_{k,l}(W'_j)\subset W'_j$ and, for every $i$, $E_{k,l}(W_i)\subset W_i$; since $f$ is $\mathfrak{gl}(V)$-equivariant, one has also $E_{k,l}(Z_{j,i})\subset Z_{j,i}$.
	
	Let us proceed by contradiction. Suppose $\psi$ is a weight for which 
	$(Z_j)^\psi \neq (Z_{j, i})^\psi$. Then $\psi = \sum a_l \varepsilon_l$ and we may suppose that $\psi$ is such that in its decomposition appear the minimum number of $\varepsilon_l$ with $l>j+q$. By the hypothesis on $i$, there is at least one $l>j+q$ such that $a_l > 0$. In this case, there is also at least one $k$ such that $j < k \leq j+q$ and $a_k = 0$. 
	
	Let $\sigma\in S'$ be the permutation exchanging $k$ and $l$. Let $\theta = \ \sigma ( \psi )$. By $S$-equivariance, one has $\dim(Z_j^{\theta}) = \dim(Z_j^{\psi})$. Furthermore, one checks that the operator $E_{k,l}^{a_k}$ sends isomorphically $(W'_j)^{\theta}$ to $(W'_j)^{\psi}$; by $\mathfrak{gl}(V)$-equivariance of $f$, it also sends isomorphically $(Z_j)^{\theta}$ to $(Z_j)^{\psi}$ and $(Z_{j, i})^{\theta}$ to $(Z_{j, i})^{\psi}$. Since in the decomposition of $\theta$ there appears one less $\varepsilon_i$ with $i>j+q$ and $a_i \neq 0$, it follows that $(Z_j)^{\theta}=(Z_{j, i})^{\theta}$ and thus also $(Z_j)^{\psi}=(Z_{j, i})^{\psi}$. The contradiction proves that $Z_j = Z_{j, i}$, and thus $f$ is strict.
	
	Let now $U$ be a closed $\mathfrak{gl}(V)$-stable subspace of $W$. Since by Lemma~\ref{lem: weightpart} one has $U = \overline{U^{wt}}$, it follows that $U$ is $S$-stable. If $U_i = U\cap W_i$, then $U=\bigcup_i U_i$ is a presentation of $U$ as inductive limit of linearly compact closed $\mathfrak{h}$-submodules. The above arguments can be repeated with $Z=f(U)$, $Z_j = Z\cap W'_j$, $Z_{j,i}=f(U_i)\cap Z_j$, to prove that $f_{|U}$ has closed image, and hence is strict.
	
	This concludes the proof.
	\end{proof}
	
	\begin{prop}\label{prop:sub}
	Every object of $\mathbf{T}_{\mathfrak{gl}(V)}$ is isomorphic to a subobject of a finite direct sum of $\mathbf{V}^{p,q}$-s.
	\end{prop}
	\begin{proof}
	Let $W = Z/U$ where $U\subset Z$ are closed $\mathfrak{gl}(V)$-equivariant subspaces of a finite direct sum $M = \bigoplus_j \mathbf{V}^{p_i,q_i} $. It follows from Lemma~\ref{T-split} that $W^{wt}\cong Z^{wt}/U^{wt}$. Furthermore, by \cite[Proposition 4.5]{DPS}, any object of $\mathbb{T}_{\mathfrak{gl}(\infty)}$ is isomorphic to a subobject of an injective object of the form $\bigoplus_j V^{r_j,s_j} $. Let $\varphi: W^{wt}\rightarrow A$ be a $\mathfrak{gl}(V)$-equivariant isomorphism, where $A$ is a $\mathfrak{gl}(V)$-stable subspace of 
	$( M' ) ^{wt}= \bigoplus_j V^{r_j,s_j} $  for $M' = \bigoplus_j \mathbf{V}^{r_j,s_j} $. Let $Z'$ be the closure of $A$ in $M'$. Moreover, by the injectivity of $( M' ) ^{wt} $ and Lemma~\ref{lem:full}, it follows that $\varphi$ is the restriction of a $\mathfrak{gl}(V)$-equivariant map $\Tilde{\varphi}: M^{wt} \rightarrow (M')^{wt}$ which extends to a 
	$\mathfrak{gl}(V)$-equivariant map $f: M \rightarrow M'$. By Proposition~\ref{prop:strict}, one has that $f(Z)$ is closed, and is thus equal to $Z'$. Now $f$ induces a bijective continuous linear map between $W$ and $Z'$, which is a topological isomorphism by Lemma~\ref{lem:iso-ind}(a). This proves the proposition.
	\end{proof}
	
	\subsection{The category $\widehat{{\mathbf T}}_{\mathfrak{gl}(V)}$ and duality}
	Finally, we define the category of topological vector representations $\widehat{{\mathbf T}}_{\mathfrak{gl}(V)}$ as the dual to the category $\mathbf{T}_{\mathfrak{gl}(V)}$.
	
	\begin{defn}
		The objects of the category $\widehat{{\mathbf T}}_{\mathfrak{gl}(V)}$ are topological vector spaces of the form $Z/U$, where $Z$ is a closed $\mathfrak{gl}(V)$-stable subspace of a finite direct sum of topological vector spaces of the form $\widehat{\mathbf{V}}^{p,q}$, and $U$ is a closed $\mathfrak{gl}(V)$-subspace of $Z$. The morphisms in the category $\widehat{\mathbf{T}}_{\mathfrak{gl}(V)}$ are $\mathfrak{gl}(V)$-equivariant continuous linear maps.
	\end{defn}

	\begin{thm}\label{thm:main}
	The following statements hold:
	\begin{enumerate}[(i)]
	    \item \label{uno}
	    The categories $\mathbf{T}_{\mathfrak{gl}(V)}$ and $\widehat{{\mathbf T}}_{\mathfrak{gl}(V)}$ are abelian subcategories of $\mathcal{I}$ and $\mathcal{P}$ respectively.
	    \item \label{due}
	    The duality between $\mathcal{I}$ and $\mathcal{P}$ restricts to a duality between $\mathbf{T}_{\mathfrak{gl}(V)}$ and $\widehat{{\mathbf T}}_{\mathfrak{gl}(V)}$.
	    \item \label{tre}
	    The weight part functor
	    $$
	    (\phantom{a})^{wt}: \mathbf{T}_{\mathfrak{gl}(V)} \rightarrow \mathbb{T}_{\mathfrak{gl}(\infty)}
	    $$
	    is an equivalence of abelian categories.
	\end{enumerate}
	\end{thm}
	\begin{proof}
	Proposition~\ref{prop:quasiab} shows that $\mathbf{T}_{\mathfrak{gl}(V)}$ and $\widehat{{\mathbf T}}_{\mathfrak{gl}(V)}$ are subcategories, respectively, of the quasi-abelian categories $\mathcal{I}$ and $\mathcal{P}$ which are closed under taking kernels and cokernels. Thus they inherit the quasi-abelian structure. Furthermore, a linear map $f$ is $\mathfrak{gl}(V)$-equivariant if and only if its dual $f^\ast$ is $\mathfrak{gl}(V)$-equivariant. This implies that the duality stated in Proposition~\ref{prop:duals} restricts to a duality between the quasi-abelian categories $\mathbf{T}_{\mathfrak{gl}(V)}$ and $\widehat{{\mathbf T}}_{\mathfrak{gl}(V)}$, and hence proves 
	(\ref{due}).
	
	To prove (\ref{uno}), observe that a quasi-abelian category is abelian if and only if every morphism is strict (see \cite{Buh}, Remark~4.7 and Remark~4.9). By duality, it suffices to prove that every morphism in $\mathbf{T}_{\mathfrak{gl}(V)}$ is strict. By Lemma~\ref{lem:strict}, this amounts to showing that every morphism of $\mathbf{T}_{\mathfrak{gl}(V)}$ has closed image. 
	
	Let $f:W\rightarrow W'$ be a morphism in $\mathbf{T}_{\mathfrak{gl}(V)}$. By Proposition~\ref{prop:sub}, one may suppose that $W$ and $W'$ are submodules of respective modules $M$ and $M'$ which are finite direct sums of $\mathbf{V}^{p,q}$-s. By \cite[Proposition~4.5]{DPS} and Lemma~\ref{lem:full}, one gets that the $\mathfrak{gl}(V)$-equivariant continuous linear maps from $W$ to $W'$ are restrictions of $\mathfrak{gl}(V)$-equivariant continuous linear maps from $M$ to $M'$, and such restrictions are strict by Proposition~\ref{prop:strict}. Thus (\ref{uno}) is proved.
	
	Let us prove (\ref{tre}). We keep the notations $W, W', M, M'$ from (i). Then one has a commutative diagram
	$$
	\begin{CD}
	\Hom_{\mathbf{T}_{\mathfrak{gl}(V)}}(M, M')@>{a}>> \Hom_{\mathbf{T}_{\mathfrak{gl}(V)}}(W, M') \\
	@V{c}VV @V{d}VV \\
	\Hom_{\mathbb{T}_{\mathfrak{gl}(\infty)}}(M^{wt}, (M')^{wt})@>{b}>> \Hom_{\mathbb{T}_{\mathfrak{gl}(\infty)}}(W^{wt}, (M')^{wt})
	\end{CD}
	$$
	The map $c$ is bijective by Lemma~\ref{lem:full}; the map $b$ is surjective by injectivity of $(M')^{wt}$ (\cite[Proposition~4.5]{DPS}). It follows that $d$ is surjective.
	Since $W^{wt}$ is dense in $W$, the map $d$ is also injective, and hence $d$ is an isomorphism.
	
	Assume $W'$  is the kernel of a map 
	$g\in \Hom_{\mathbf{T}_{\mathfrak{gl}(V)}}(M', N')$, where $N'$ is a finite direct sum of $\mathbf{V}^{p,q}$-s. This leads to the commutative diagram
	$$
	\begin{CD}
	0 @>>>\Hom_{\mathbf{T}_{\mathfrak{gl}(V)}}(W, W')@>>> \Hom_{\mathbf{T}_{\mathfrak{gl}(V)}}(W, M') @>>> \Hom_{\mathbf{T}_{\mathfrak{gl}(V)}}(W, N')\\
	@. @VVV @V{d}VV  @V{\cong}VV\\
	0 @>>>\Hom_{\mathbb{T}_{\mathfrak{gl}(\infty)}}(W^{wt}, (W')^{wt})@>>> \Hom_{\mathbb{T}_{\mathfrak{gl}(\infty)}}(W^{wt}, (M')^{wt}) @>>> \Hom_{\mathbb{T}_{\mathfrak{gl}(\infty)}}(W^{wt}, (N')^{wt})
	\end{CD}
	$$
	where the rows are exact and the middle vertical arrow is the map $d$ from above. We have established that $d$ is an isomorphism, and in the same way one can establish that the third vertical arrow is an isomorphism. It thus follows that the first vertical arrow is also an isomorphism. This proves that the functor $(\phantom{a})^{wt}$ is fully faithful. 
	
	Let us prove the essential surjectivity of the functor $(\phantom{a})^{wt}$. Let $A$ be an object of $\mathbb{T}_{\mathfrak{gl}(\infty)}$. By \cite[Proposition~4.5]{DPS}, $A$ may be assumed to be the kernel of a map $\phi\in \Hom_{\mathbb{T}_{\mathfrak{gl}(\infty)}}((M'')^{wt}, (M''')^{wt})$, where $M'', M'''$ are finite direct sums of $\mathbf{V}^{p,q}$-s. By Lemma~\ref{lem:full}, the map $\phi$ is the restriction of a unique map $\psi\in \Hom_{\mathbf{T}_{\mathfrak{gl}(V)}}(M'', M''')$. One has then $(\mathrm{ker}\psi)^{wt}=A$; thus the functor $(\phantom{a})^{wt}$ is essentially surjective. This concludes the proof of (\ref{tre}).
	\end{proof}

	\begin{cor} \label{quot}
	    Every object of $\widehat{\mathbf{T}}_{\mathfrak{gl}(V)}$ is isomorphic to a quotient of a finite direct sum of $\widehat{\mathbf{V}}^{p,q}$-s.
	\end{cor}
	\begin{proof}
	    Follows from Theorem~\ref{thm:main}(\ref{due}) and Proposition~\ref{prop:sub}.
	\end{proof}

	\section{Monoidal structure and further properties}
	
	We now endow the categories $\widehat{\mathbf{T}}_{\mathfrak{gl}(V)}$ and $\mathbf{T}_{\mathfrak{gl}(V)}$ with symmetric monoidal structures. We use Beilinson's tensor product operations between topological vector spaces introduced in \cite{Bei}.
	We recall these operations here, following Positselski's notation \cite{Po}, and show that they endow the categories $\widehat{\mathbf{T}}_{\mathfrak{gl}(V)}$ and $\mathbf{T}_{\mathfrak{gl}(V)}$ with antiequivalent symmetric monoidal structures.
	
	The following definitions are taken from  Sections 12 and 13 of \cite{Po}. Let $W_1$, $W_2$ be vector spaces endowed with linear topologies.
	
	The topological vector space $W_1 \otimes^\ast W_2$ is the usual tensor product $W_1 \otimes W_2$ endowed with the linear topology for which a subspace $E\subset W_1 \otimes W_2$ is open if the following properties hold:
	\begin{itemize}
	    \item there exist open subspaces $P_1 \subset W_1$ and $P_2 \subset W_2$ such that $P_1 \otimes P_2 \subset E$;
	    \item for every $w_1 \in W_1$, there is an open subspace $Q_{w_1}\subset W_2$ such that 
	    $w_1 \otimes Q_{w_1} \subset E$;
	    \item for every $w_2 \in W_2$, there is an open subspace $Q_{w_2}\subset W_1$ such that 
	    $Q_{w_2}\otimes w_2 \subset E$.
	\end{itemize}

	The topological vector space $W_1 \otimes^! W_2$ is the usual tensor product $W_1 \otimes W_2$ endowed with the linear topology for which a subspace $E\subset W_1 \otimes W_2$ is open if there exist open subspaces $P_1 \subset W_1$ and $P_2 \subset W_2$ such that $P_1 \otimes W_2 + W_1 \otimes P_2 \subset E$.
	
	The completions of the vector spaces $W_1 \otimes^\ast W_2$ and $W_1 \otimes^! W_2$ are denoted respectively by $W_1 \widehat{\otimes}^\ast W_2$ and $W_1 \widehat{\otimes}^! W_2$.
	
	For the convenience of the reader, we recall here some facts about the operations $\widehat{\otimes}^\ast$ and $\widehat{\otimes}^!$.
	
	\begin{lem} \label{lem:t1}
	The following statements hold:
	\begin{enumerate}
	    \item \label{one}
	    For $W$ discrete and any $W'$, one has 
	    $W\widehat{\otimes}^\ast W' = W\otimes W'$ where $W\otimes W'$ is endowed with the inductive limit topology.
	    
	    \item \label{two} For $W$ discrete and $W' = \varprojlim W'_j$ linearly compact, one has 
	    $$
	    W \widehat{\otimes}^! W' = \varprojlim  \ W\otimes W'_j \ .
	    $$
	    
	    \item \label{three}
	    For $W = \varprojlim W_j$ and $W' = \varprojlim W'_j$ linearly compact, one has 
	    $$W\widehat{\otimes}^\ast W' = W\widehat{\otimes}^! W' = \varprojlim \ W_j\otimes W'_j \ . $$
	    
	    \item \label{four}
	    For any $W_1 , W_2, W_3$, one has canonical isomorphisms  
	    $$(W_1 \widehat{\otimes}^\ast W_2) \widehat{\otimes}^\ast W_3 \cong W_1 \widehat{\otimes}^\ast (W_2 \widehat{\otimes}^\ast W_3)$$ 
	    and 
	    $$(W_1 \widehat{\otimes}^! W_2) \widehat{\otimes}^! W_3 \cong W_1 \widehat{\otimes}^! (W_2 \widehat{\otimes}^! W_3) \ . $$

	\end{enumerate}
	\end{lem}
	\begin{proof}
	    Statements (1) and (3) follow from \cite[Examples~13.1 (1), (2)]{Po}. Statement (2) is observed in \cite[Remark 12.1]{Po} and
	    statement (4) is proved in \cite[Proposition 13.4]{Po}.
	\end{proof}
	
	The above facts lead to the following
	
	\begin{lem}\label{lem:t2}
	One has canonical isomorphisms:
	
	\begin{enumerate}
	    \item[(a)]\label{aa}
	    $\mathbf{V}^{p,q} = V^{\widehat{\otimes}^\ast p}\ \widehat{\otimes}^\ast \  \mathbf{V}^{\widehat{\otimes}^\ast q}$,
	    \item[(b)] \label{bb}
	    $\widehat{\mathbf{V}}^{p,q} = V^{\widehat{\otimes}^! p}\ \widehat{\otimes}^! \  \mathbf{V}^{\widehat{\otimes}^! q}$,
	    \item[(c)] \label{cc}
	    $V\widehat{\otimes}^! \mathbf{V} \cong \mathfrak{gl}(V)$.
	\end{enumerate}
	\end{lem}
	\begin{proof}
	    The space $V = \bigcup_i V_i$ is discrete and the space $\mathbf{V} = \varprojlim (V_i)^\ast$ is linearly compact, hence 
	    $V^{\otimes p} = V^{\widehat{\otimes}^\ast p}$ by Lemma~\ref{lem:t1}(1), and 
	    $(V^{\otimes q})^\ast = \varprojlim (V_i^{\otimes q})^\ast = \mathbf{V}^{\widehat{\otimes}^! q}$ by Lemma~\ref{lem:t1}(3).
	    
	    Furthermore, by statement (2) of Lemma~\ref{lem:t1}, one has 
	    $$
	    V \widehat{\otimes}^! \mathbf{V} = \varprojlim \  V \otimes V_i^\ast = \varprojlim \ \mathrm{Hom}(V_i , V) = \mathrm{End}V = \mathfrak{gl}(V) \ .
	    $$
	\end{proof}
	
	We can now strengthen Theorem~\ref{thm:main} as follows.
	\begin{prop} \label{prop:t}
	    The following statements hold:
	    \begin{enumerate}
	        \item[(i)] The category $\mathbf{T}_{\mathfrak{gl}(V)}$ is symmetric monoidal when endowed with the tensor product $\widehat{\otimes}^\ast$.
	        
	        \item[(ii)] The category $\widehat{\mathbf{T}}_{\mathfrak{gl}(V)}$ is symmetric monoidal when endowed with the tensor product $\widehat{\otimes}^!$.
	        
	        \item[(iii)] The functor $(\phantom{a})^{wt} : \mathbf{T}_{\mathfrak{gl}(V)} \rightarrow \mathbb{T}_{\mathfrak{gl}(\infty)}$ establishes an equivalence of abelian monoidal categories.
	        
	        \item[(iv)] The functor $(\phantom{a})^\ast : \mathbf{T}_{\mathfrak{gl}(V)} \rightarrow \widehat{\mathbf{T}}_{\mathfrak{gl}(V)}$ establishes an antiequivalence of abelian monoidal categories.
	    \end{enumerate}
	\end{prop}
	
	\begin{proof}
	Let us first show that if $W$ and $W'$ are two objects of $\mathbf{T}_{\mathfrak{gl}(V)}$, then $W\widehat{\otimes}^\ast W'$ is an object of $\mathbf{T}_{\mathfrak{gl}(V)}$. Indeed, by Proposition~\ref{prop:sub} and by Lemma~\ref{lem:ind-linsubquot}, we may assume that there are  (in general not $\mathfrak{gl}(V)$-stable) objects $C$ and $C'$ of $\mathcal{I}$ such that $W\oplus C$ and $W' \oplus C'$ are isomorphic to finite direct sums of $\mathbf{V}^{p,q}$-s. Since the tensor product $\widehat{\otimes}^\ast$ is a biadditive functor, Lemma~\ref{lem:t2}(a) implies that $W\widehat{\otimes}^\ast W'$ is isomorphic to a closed $\mathfrak{gl}(V)$-stable subspace of a finite direct sum of $\mathbf{V}^{p,q}$-s; hence it is an object of $\mathbf{T}_{\mathfrak{gl}(V)}$. 
	
	The fact that $(\mathbf{T}_{\mathfrak{gl}(V)}, \widehat{\otimes}^\ast)$ is a symmetric monoidal category
	follows from the associativity in Lemma~\ref{lem:t1}(\ref{four}) and the commutativity of $\widehat{\otimes}^\ast$. This proves (i).
	Statement (ii) is proved analogously by using Lemma~\ref{lem:t2}(b), Corollary~\ref{quot} and the associativity and commutativity of $\widehat{\otimes}^!$.
	
	To prove (iii), observe that Lemma~\ref{lem:t2}(a) implies
	$$(\mathbf{V}^{p,q} \widehat{\otimes}^\ast \mathbf{V}^{p',q'})^{wt} = (\mathbf{V}^{p,q})^{wt }\otimes (\mathbf{V}^{p',q'})^{wt} \ . $$
	Using Lemma~\ref{T-split} and Proposition~\ref{prop:sub}, one gets  decompositions of $\mathfrak{h}$-modules $W\oplus U = \bigoplus_i \mathbf{V}^{p_i,q_i}$ and $W'\oplus U' = \bigoplus_j \mathbf{V}^{p'_j,q'_j}$ for any two objects $W$ and $W'$  of of $\mathbf{T}_{\mathfrak{gl}(V)}$. 
	This shows that $(W\widehat{\otimes}^\ast W')^{wt} = W^{wt}\otimes (W')^{wt}$, and (iii) is proved.
	
	Let us now establish (iv). Indeed, by Lemma~\ref{lem:t2}(a), (b) we have
	$$(\mathbf{V}^{p,q} \ \widehat{\otimes}^\ast \ \mathbf{V}^{p',q'})^\ast = (\mathbf{V}^{p+p',q+q'})^\ast
	= \widehat{\mathbf{V}}^{q+q', p+p'} = \widehat{\mathbf{V}}^{q, p} \ \widehat{\otimes}^! \  \widehat{\mathbf{V}}^{q', p'}
	= (\mathbf{V}^{p,q})^\ast \ \widehat{\otimes}^! \ (\mathbf{V}^{p',q'})^\ast \ . $$
	Therefore, using the semisimplicity of $\mathcal{I}$ and Proposition~\ref{prop:sub}, we conclude that 
	$$(W\ \widehat{\otimes}^\ast \ W')^\ast = W^\ast \ \widehat{\otimes}^! \  (W')^\ast$$ 
	for any two objects $W, W'$ of $\mathbf{T}_{\mathfrak{gl}(V)}$. This proves (iv).
	\end{proof}

	To state some corollaries of Proposition~\ref{prop:t}, we need to fix some notation. Recall that the \textit{radical} $ \mathrm{rad}M $ of an object $ M $ of an abelian category is the intersection of all kernels of all homomorphisms of $ M $ into semisimple objects. Then the descending \textit{radical filtration} of $ M $ is
	$$
	\mathrm{rad}^{i}M = \mathrm{rad}(\mathrm{rad}^{i-1}M)
	$$
	for $ i\geqslant 2$, and the \textit{layers} $ \mathrm{rad}^{i}M/\mathrm{rad}^{i+1}M$ of this filtration are semisimple objects. Moreover, if $ M $ has finite length (i.e., admits a Jordan-H\"older series), then the radical filtration is finite and separating.
	
	Given a Young diagram $ \lambda $, by $ |\lambda|  $ we denote its degree and $ \lambda^{\perp} $ stands for the conjugate diagram. Next, by $ N^{\lambda}_{\mu, \nu} $ we denote the Littlewood-Richardson coefficient associated to the Young diagrams $ \lambda, \mu, \nu $.
	
	Finally, we note that the definition of Schur functors $ \mathbb{S}_{\lambda} $ for Young diagrams $ \lambda $ makes sense for pure tensors also in the category $ \widehat{\mathbf{T}}_{\mathfrak{gl}(V)} $. Indeed, one can apply the usual symmetrization procedure associated with a Young tableau of shape $ \lambda $ to tensor products of the form $ V^{\otimes p}$ and obtain $ \mathbb{S}_{\lambda}(V) $. Then one sets $ \widehat{\mathbb{S}}_{\lambda}(\mathbf{V}):= \mathbb{S}_{\lambda}(V)^{*} $. Moreover, $ \mathbb{S}_{\lambda}(V) $ and $ \widehat{\mathbb{S}}_{\lambda}(\mathbf{V}) $ are irreducible objects of $ \widehat{\mathbf{T}}_{\mathfrak{gl}(V)} $ for any $ \lambda $.
	
	Theorem~\ref{thm:main} has now the following 
	\begin{cor}\phantom{text}\\
		\begin{itemize}
			\item[a)] The simple objects of $ \widehat{\mathbf{T}}_{\mathfrak{gl}(V)} $ are parametrized by pairs of Young diagrams $ (\lambda, \mu)\colon$ the simple object $\widehat{\mathbf{V}}_{\lambda, \mu} $ is the quotient of the object $ \mathbb{S}_{\lambda}(V)\ \widehat{\otimes}^! \ \widehat{\mathbb{S}}_{\mu}(\mathbf{V})\in \widehat{\mathbf{T}}_{\mathfrak{gl}(V)} $ by its radical. Moreover, the object $ \mathbb{S}_{\lambda}(V)\ \widehat{\otimes}^! \ \widehat{\mathbb{S}}_{\mu}(\mathbf{V})$ has finite length.
			\item[b)] The objects $ \mathbb{S}_{\lambda}(V)\ \widehat{\otimes}^! \ \widehat{\mathbb{S}}_{\mu}(\mathbf{V})\in \widehat{\mathbf{T}}_{\mathfrak{gl}(V)} $, for pairs of Young diagrams $ (\lambda, \mu) $, are projective and indecomposable, and any indecomposable projective object of the category $ \widehat{\mathbf{T}}_{\mathfrak{gl}(V)} $ is isomorphic to $ \mathbb{S}_{\lambda}(V)\ \widehat{\otimes}^! \ \widehat{\mathbb{S}}_{\mu}(\mathbf{V})$ for some $ (\lambda, \mu) $.
			\item[c)] The category $ \widehat{\mathbf{T}}_{\mathfrak{gl}(V)} $ is equivalent to the category of locally unitary finite-dimensional modules over the infinite-dimensional Koszul algebra $ \mathcal{A}_{sl(\infty)} $ studied in \cite[Sect. 4]{DPS}.
			\item[d)] We have 
			$$ \mathrm{rad}^{i}\left(\mathbb{S}_{\lambda}(V)\ \widehat{\otimes}^! \ \widehat{\mathbb{S}}_{\mu}(\mathbf{V})\right)/\mathrm{rad}^{i+1}\left(\mathbb{S}_{\lambda}(V)\ \widehat{\otimes}^! \ \widehat{\mathbb{S}}_{\mu}(\mathbf{V})\right) \simeq \bigoplus_{\nu, \varkappa}\left(\sum_{|\gamma| = i}N^{\lambda}_{\nu, \gamma}N^{\mu}_{\varkappa, \gamma}\right)\widehat{\mathbf{V}}_{\nu, \varkappa} \ , 
			$$
			where $\nu$, $\varkappa$ and $\gamma$ are Young diagrams.
			In particular, if $ \widehat{\mathbf{V}}_{\nu, \varkappa} $ is a simple constituent of $ \mathrm{rad}^{i}\left(\mathbb{S}_{\lambda}(V)\ \widehat{\otimes}^! \ \widehat{\mathbb{S}}_{\mu}(\mathbf{V})\right)/\mathrm{rad}^{i+1}\left(\mathbb{S}_{\lambda}(V)\ \widehat{\otimes}^! \ \widehat{\mathbb{S}}_{\mu}(\mathbf{V})\right)  $ then the Young diagrams $ \nu $ and $ \varkappa $ are obtained respectively from $ \lambda $ and $ \mu $ by removing exactly $ i $ boxes.
			\item[e)] For any $ (\lambda, \mu) $, $ (\nu, \varkappa) $, and $ i\in\mathbb{Z}_{\geqslant 0} $, the vector space $ \mathrm{Ext}^{i}_{{\widehat{\mathbf{T}}}_{\mathfrak{gl}(V)}}\left(\widehat{\mathbf{V}}_{\lambda, \mu}, \widehat{\mathbf{V}}_{\nu, \varkappa}\right) $ is finite dimensional and its dimension equals the multiplicity of $ \widehat{\mathbf{V}}_{\nu, \varkappa^{\perp}} $ in 
			$$ \mathrm{rad}^{i}\left(\mathbb{S}_{\lambda}(V)\ \widehat{\otimes}^! \ \widehat{\mathbb{S}}_{\mu^{\perp}}(\mathbf{V})\right)/\mathrm{rad}^{i+1}\left(\mathbb{S}_{\lambda}(V)\ \widehat{\otimes}^! \ \widehat{\mathbb{S}}_{\mu^{\perp}}(\mathbf{V})\right),  
			$$ 
			i.e., 
			$$ 
			\mathrm{dim}\mathrm{Ext}^{i}_{{\widehat{\mathbf{T}}}_{\mathfrak{gl}(V)}}\left(\widehat{\mathbf{V}}_{\lambda, \mu}, \widehat{\mathbf{V}}_{\nu, \varkappa}\right) = \sum_{|\gamma| = i}N^{\lambda}_{\nu, \gamma}N^{\mu^\perp}_{\varkappa^\perp, \gamma} \ . 
			$$
		\end{itemize}
	\end{cor}
	\begin{proof}
	    All statements follow from the existence of an antiequivalence of the symmetric monoidal categories $\widehat{\mathbf T}_{\mathfrak{gl}(V)}$ and ${\TT}_{\mathfrak{gl}(\infty)}$ (Proposition~\ref{prop:t}), and from the respective results from \cite{DPS} concerning the category ${\TT}_{\mathfrak{gl}(\infty)}$.
	\end{proof}
	
	We conclude the paper by presenting briefly a universality property of the symmetric monoidal category $ \widehat{\mathbf{T}}_{\mathfrak{gl}(V)} $.
	
	Recall that to every object $ X $ of the category $ \TT_{\mathfrak{gl}(\infty)} $ one assigns its dual $ X_{*} $ and that there is a canonical morphism surjective $ X\otimes X_{*}\to\CC $. The category $ \TT_{\mathfrak{gl}(\infty)} $ is not a rigid symmetric monoidal category according to the definition in \cite{EG}, since the morphism $ X\otimes X_{*}\to\CC $ admits in general no splitting $ \CC\to X\otimes X_{*} $.
	
	In comparison, to every object $ Y $ of the category $ \widehat{\mathbf{T}}_{\mathfrak{gl}(V)} $ one assigns a dual $ Y^{\vee} $ defined as $ ((Y^{*})_{*} )^{*}$, where the functor $ (\phantom{a})_{*}\colon\mathbf{T}_{\mathfrak{gl}(V)}\to \mathbf{T}_{\mathfrak{gl}(V)} $ is transferred to the category $ \mathbf{T}_{\mathfrak{gl}(V)} $ via the equivalence $ \mathbf{T}_{\mathfrak{gl}(V)}\overset{(\phantom{a})^{wt}}{\to}\TT_{\mathfrak{gl}(\infty)}$. However, since $ \widehat{\mathbf{T}}_{\mathfrak{gl}(V)} $ is antiequivalent to $ \TT_{\mathfrak{gl}(\infty)} $, there is an injective morphism $ \CC\to X\widehat{\otimes}^! X^{\vee} $ which does not split in general.
	
	\begin{cor}
		Let $ (\mathcal{T}, \otimes) $ be an abelian linear symmetric monoidal category. Assume that two objects $ X, Y\in\mathcal{T} $ and nonzero morphism $\mathbf{1}\to X\otimes Y $  are given, where $\mathbf{1}$ is the monoidal unit of $ \mathcal{T} $. Then there exists a right-exact monoidal functor 
		$$
		F\colon \widehat{\mathbf{T}}_{\mathfrak{gl}(V)}\to \mathcal{T},
		$$
		satisfying  $ F(V) = X $, $ F(\mathbf{V}) = Y$, and sending the injection $ \CC\mathrm{Id}\to \mathfrak{gl}(V) $ to the monomosphism $ \mathbf{1}\to X\otimes Y$.
	\end{cor}

\end{document}